\documentclass[11pt]{amsart}
\usepackage{amsmath,amsfonts,comment,amsthm,euscript,dirtytalk,mathtools,amssymb,graphicx,mathrsfs,enumitem, cleveref, tikz, pgfplots}
\pgfplotsset{compat=1.8}
\usepackage{StyleSheet}

\begin{document}
\title{Sharp Eigenfunction Bounds on the Torus for large $p$}
\author{Daniel Pezzi}
\date{}

\begin{abstract}
We prove the discrete restriction conjecture holds with no loss when $p>\frac{2d}{d-4}$ and $d\geq 5$. That is, we show optimal $L^p$ bounds for eigenfunctions of the Laplacian on the square torus for large values of $p$. This improves the results of Bourgain and Demeter. Our proof method is a refinement of the circle method approach previously used to establish results with a subpolynomial loss. This represents the first sharp $L^p$ bounds for eigenfunctions on the torus since the work of Cooke and Zygmund.

We present applications to bounds for spectral projectors and the additive energy of integer lattice points on higher dimensional spheres. These results are similarly sharp. We also prove results with a logarithmic loss that hold in a wider range of $p$.
\end{abstract}
\maketitle

\maketitle

\newcommand{\Addresses}{{
  \bigskip
  \footnotesize
  Daniel~Pezzi; \textsc{Department of Mathematics, Krieger Hall, Homewood Campus, Johns Hopkins University, 3400 N. Charles Street Baltimore, MD 21218, United States}\par\nopagebreak
  \textit{E-mail address}, \texttt{dpezzi1@jh.edu}

}}

\section{Introduction}\label{Section1}
\subsection{Background}
Let $\T{d} = \R{d}/\Z{d}$ be the square torus. We may view functions on this manifold as $\Z{d}$-periodic functions on $\R{d}$, whose values are completely determined by their values on the fundamental domain $[0,1]^d$. If $f\in L^2(\T{d})$, it is well known we may expand this function in Fourier series.

\begin{equation}
    f(x) = \sum_{\mathbf{k}\in\Z{d}}\mathcal{F}(f)(\mathbf{k})e^{2\pi i \mathbf{k} \cdot x}.
\end{equation}
The plane waves $e^{2\pi i \mathbf{k}\cdot x}\coloneq e(\mathbf{k}\cdot x)$ are eigenfunctions\footnote{We use $\mathbf{k}$ to represent integer vectors in $\Z{d}$. For other notional conventions, see \Cref{1: Subsection Notation}} of the square root of the negative Laplacian $\sqrt{-\Delta}$ with eigenvalue $2\pi |\mathbf{k}|$. It would be equivalent in what follows to consider eigenfunctions of $-\Delta$, although statements would have altered numerology. Let $P_N$ be the projection onto the eigenspace of  $\sqrt{-\Delta}$ with eigenvalue $2\pi N$. Then 

\begin{equation}\label{1: eSpaceProjectorDefEqn}
    e_N(x) \coloneq P_N f(x) = \sum_{|\mathbf{k}|=N}\mathcal{F}(f)(\mathbf{k})e^{2\pi i \mathbf{k} \cdot x}.
\end{equation}
That is, we project onto integer lattice point frequencies that lie on the sphere of radius $N$. We will also use the notation $\la = N^2$. Bourgain in \cite{bourgainTorusEFunc1993} asked how large such functions can be in $L^p$.

\begin{conjecture}[The Discrete Restriction Conjecture]\label{1: DiscreteRestrictionConj}
    Let $p\geq 2$ and $d\geq 2$. Let $\T{d}$ be the square torus. Suppose $e_N$ is an eigenfunction of  $\sqrt{-\Delta}$. Let $\eps>0$. Then 

    \begin{equation}\label{1: DiscreteRestrictionEqn}
        \lpnorm{e_N}{p}{\T{d}}\lesssim_{\eps} N^\eps(N^{\frac{d-2}{2}-\frac{d}{p}}+1)\lpnorm{e_N}{2}{\T{d}}.
    \end{equation}
    The extra $C_\eps N^\eps$ factor can be omitted if $d\geq 5$.
\end{conjecture}
See \Cref{1: Subsection DiscRestriction} for the intuition for this conjecture, which heavily involves the number of lattice points on spheres with integer radius. Let

\begin{equation*}
    \mathcal{F}_{d,N}\coloneq \{\mathbf{k}=(k_1,...,k_d)\in\Z{d}: |\mathbf k|=N\}.
\end{equation*}
It is well known that the following bound hold for the cardinality of integer lattice points on the sphere.

\begin{equation}
    \# \mathcal{F}_{d,N} \lesssim_\eps N^\eps N^{d-2+\eps}.
\end{equation}
Lattice points on the sphere are more regularly distributed as the dimension increases. The $\eps$ factor may be removed when $d\geq 5$ (which is why the same statement holds for \Cref{1: DiscreteRestrictionConj}). The upper bound may also be replaced by an asymptotic when $d\geq 5$ as every integer $N$ corresponds to a non-empty eigenspace with the given multiplicity. 

\Cref{1: DiscreteRestrictionConj} was proven with a factor of $C_\eps N^\eps$ in all dimensions when $p\leq \frac{2(d+1)}{d-1}$ by the $\ell^2$-decoupling theorem of Bourgain and Demeter \cite{BourgainDemeter2015}. It had been previously proven for $p\geq \frac{2d}{d-3}$ by the same authors in \cite{bourgainDemeterDiscreteRestrictImprovements2013}, also with an $\eps$-loss. Combining the results of these two papers yields the conjecture with loss when $p\geq \frac{2(d-1)}{d-3}$. Our main result is that this loss can be removed for large $p$.

\begin{theorem}\label{1: SharpMainResultThm}
    Let $d\geq 5$ and $p>\frac{2d}{d-4}$. Let $\T{d}$ be the square torus. Then for every eigenfunction $e_N$, the following bound holds.

        \begin{equation}\label{1: DiscreteRestrictionEqn}
        \lpnorm{e_N}{p}{\T{d}}\lesssim N^{\frac{d-2}{2}-\frac{d}{p}}\lpnorm{e_N}{2}{\T{d}}.
    \end{equation}
\end{theorem}
This result contains the first sharp bounds for toral eigenfunctions since the work of Cooke and Zygmund. Cooke proved the result in the course of showing a Cantor-Lebesgue Theorem in \cite{cooke1971}. The result was studied on its own by Zygmund in \cite{Zygmund1974} to answer a question of Fefferman. Both proved the sharp result with no loss when $d=2$ and $2\leq p \leq 4$ using essentially the same geometric argument. Our proof method differs greatly from that work. In particular, we refine the circle method approach used in \cite{bourgainTorusEFunc1993} and \cite{bourgainDemeterDiscreteRestrictImprovements2013}. For this reason, our results are limited to the square torus. As an aside, other works often refer to bounds with a power of $N^\eps$ as \say{sharp}. In this paper, we will reserve sharp to mean that no such factor is present, although, for low dimensions and large $p$, the factor is necessary.

We also able to improve the $N^\eps$ loss to a logarithmic loss in a larger range of $p$.

\begin{theorem}\label{1: logMainResultThm}
    Let $d\geq 6$ and $p>\frac{2d}{d-3}$ or $d=5$ and $p>6$. Let $\T{d}$ be the square torus. Then, for every eigenfunction $e_N$, the following bound holds.

        \begin{equation}\label{1: DiscreteRestrictionEqn}
        \lpnorm{e_N}{p}{\T{d}}\lesssim (\log N)^{\frac{d-2}{p(d-4)}}N^{\frac{d-2}{2}-\frac{d}{p}}\lpnorm{e_N}{2}{\T{d}}.
    \end{equation}
\end{theorem}

\subsection{Organization of the paper}
In \Cref{1: Subsection Applications} we present some immediate corollaries of our main theorem which may be of independent interest. In \Cref{1: Subsection DiscRestriction} we give a tour of the discrete restriction conjecture that intuitively establishes the numerology of the conjecture and when a $C_\eps N^\eps$ factor is necessary. \Cref{1: Subsection Notation} contains the notation we use.

In \Cref{Section2} we present the proof of \Cref{1: logMainResultThm} and in \Cref{Section3} we present the proof of \Cref{1: SharpMainResultThm}. These sections are written to be largely independent of each other. Throughout both proofs we use certain known exponential sum estimates. We sketch their proofs in \Cref{SectionExpSumEstimates}. As our proof is broadly a refinement of the work of Bourgain and Demeter, specifically \cite{bourgainDemeterDiscreteRestrictImprovements2013}, we take \Cref{SectionComparison} to list the differences between this proof and theirs.

\subsection{Some Applications}\label{1: Subsection Applications}
We present two applications of our main results: bounds for approximate eigenfunctions and bounds for the additive energy of spheres.

Instead of studying $L^2\rightarrow L^p$ bounds for projections onto eigenspaces, one could instead consider projections onto spectral bands. This is equivalent to estimating the $L^p$ norm of $L^2$-normalized quasimodes. For $1\geq \delta \gtrsim \la^{-1}$, define

\begin{equation}\label{1: spectralProjectorDef}
    P_{N,\delta}f(x) = \sum_{\mathbf{k}\in A_{N,\delta}}\mathcal{F}(f)(\mathbf{k})e(\mathbf{k}\cdot x).
\end{equation}
Here $A_{N,\delta}$ is an annulus of radius $N$ and width $\delta$. There is no benefit to projecting onto smaller windows as, when $\delta \leq \frac{1}{4}\la^{-1}$, the projector equals $P_N$ defined in \eqref{1: eSpaceProjectorDefEqn}. Germain and Myerson asked for the sharp $L^2\rightarrow L^p$ bounds for this object in \cite{germainMyersonSpecProj2022}, after the problem was first studied by Hickman \cite{Hickman2020}. Sharp estimates for $p\leq \frac{2(d+1)}{d-1}$ for $\delta$ away from the eigenfunction case were proven in \cite{pezzi2025} for all tori, not just square ones.

We confirm the sharp version of the conjecture of Germain and Myerson in the analogous range.

\begin{theorem}\label{1: specProjCorThm}
    Let $d\geq 5$ and $\T{d}$ be the square torus. Let $p>\frac{2d}{d-4}$ and let $P_{N,\delta}$ be the spectral projector with $1\geq \delta \gtrsim \la^{-1}$. Then

    \begin{equation}\label{1: specProjCorEqn}
        \lpnorm{P_{N,\delta}f}{p}{\T{d}}\lesssim N^{\frac{d-1}{2}-\frac{d}{p}}\delta^{1/2}\lpnorm{f}{2}{\T{d}}.
    \end{equation}
\end{theorem}
As a note of comparison, $\la$ in the notation of \cite{germainMyersonSpecProj2022} is $N$ here. The proof is a straightforward application of $TT^*$ and crucially used that $p>\frac{2d}{d-2}$; see \cite{demeterGermain2024} Lemma 5.1 for a proof that \Cref{1: SharpMainResultThm} implies \Cref{1: specProjCorThm}. That proof is given in dimension $2$ but holds in any dimension.

It is well known that discrete restriction estimates imply bounds on the additive energy of lattice points on spheres. This connection was used to improve bounds outside of the range of $p$ for which the conjecture with $\epsilon$-loss in known by Mudgal \cite{mudgal2022}. Also, see that paper for an in depth exploration of the connection between additive combinatorics and the discrete restriction conjecture. 

For any set $A\subset \Z{d}$, we define its additive energy to be

\begin{equation}\label{1: additiveEnergyDef}
    E_{n}(A) = \#\{(a_1,...,a_{2n})\in (A)^{2n}: a_1+...+a_n = a_{n+1}+...+a_{2n}\}.
\end{equation}
Additive energies capture the additive structure of a set. Approximate groups maximize the additive energy for sets of a fixed size. For $d\geq 5$, we give sharp results for the additive energy of spheres.

Before stating our results, we quickly derive the conjectured bound. For $d\geq 5$, we have $\#\mathcal{F}_{d,N}\sim N^{d-2}$. Consider $n\mathcal{F}_{d,N}$, the set formed by adding $n$ elements of $\mathcal{F}_{d,N}$. Each element in this set has magnitude $\lesssim_n N$, and $\#n\mathcal{F}_{d,N}\lesssim N^{n(d-2)}$. If the sumsets equidistributed over lattice points in a ball of radius $\lesssim_n N$, each point would have $\sim N^{n(d-2)-d}$ $n$-tuples that sum to that given point. An application of Cauchy-Schwarz gives, for $d\geq 5$, the lower bound

\begin{equation*}
    E_{n}(\mathcal{F}_{N,d})\gtrsim N^{2n(d-2)-d}.
\end{equation*}
One conjectures that this lower bound is in fact an upper bound. We confirm this conjecture for large $d$ and $p$, using the well known connection between $L^p$ norms for even integers and additive combinatorics.

\begin{theorem}\label{1: AdditiveEnergyTheorem}
    We have $ E_{n}(\mathcal{F}_{N,d})\sim N^{2n(d-2)-d}$ when

    \begin{enumerate}
        \item $d\geq 9$.
        \item $d = 7,8$ and $n\geq 3$.
        \item $d=6$ and $n\geq 4$.
        \item $d=5$ and $n\geq 6$.
    \end{enumerate}
    Additionally, an upper bound holds with a $(\log N)^{\frac{d-2}{d-4}}$ factor if 

    \begin{enumerate}
        \item $d=7,8$ and $n= 2$.
        \item $d = 6$ and $n=3$.
        \item $d=5$ and $n= 4,5$.
    \end{enumerate}
\end{theorem}
\noindent
To prove this, one can directly compute that

\begin{equation*}
    E_{n}(\mathcal{F}_{N,d}) = \lpnorm{K}{2n}{\T{d}}^{2n}.
\end{equation*}
Here, $K(x)$ is the convolution kernel of $P_N$, which is given by \eqref{1: spectralProjectorDef} in the case $\mathcal{F}(f)(\mathbf{k})= 1$ for every $\mathbf{k}$. The result then follows from an application of \Cref{1: SharpMainResultThm} or \Cref{1: logMainResultThm}, depending on which is applicable in the given range of $p$ and $d$.

\subsection{A quick tour of Discrete Restriction}\label{1: Subsection DiscRestriction}
To explain the intuition behind \Cref{1: DiscreteRestrictionConj}, we will present the more general concept of constructive interference. Let $\Xi$ be a set of frequencies $\xi$ such that $|\xi|\sim N$ and $\#\Xi = M$. Consider the sum

\begin{equation*}
    f(x) = \sum_{\xi\in \Xi}e(\xi \cdot x).
\end{equation*}
Clearly, $f(0)=M$. By continuity, we have $|f(0)|\geq M/2$ in a ball of radius $\sim N^{-1}$ centered at $0$. We then have

\begin{equation*}
    \int_{B(0,1)}|f(x)|^pdx \gtrsim M^pN^{-d}.
\end{equation*}
The constructive interference of the exponentials, which is a manifestation of the uncertainty principle, forces a function that is very large on a small neighborhood of a point. In the case of the torus, when $d\geq 5$, we have $M\sim N^{d-2}$. The same computation then gives the first bound on the right hand side of \Cref{1: DiscreteRestrictionConj}. When $d=2,3$ or $4$, we have a factor of $C_\eps N^\epsilon$.

This point focusing example is a general phenomena that saturates $L^p$ norms of eigenfunctions (or, more generally, quasimodes). However, the two terms in the discrete restriction conjecture balance at $p=\frac{2d}{d-2}$. On generic manifolds, we expect the critical value to be $p=\frac{2(d+1)}{d-1}$. The difference is due to the degenerate number of lattice points on spheres as compared to arbitrary $1$-separated sets, which gives a global count of $N^{d-2}$ as opposed to $N^{d-1}$.

The second term on the right hand side of \Cref{1: DiscreteRestrictionConj} can be seen by considering a single plane wave. For an example with full support, one can generate a host of random examples analyzed with Khintchine's inequality that experience maximum destructive interference. The differences in the saturating examples lead to the differences in techniques used for $p>\frac{2d}{d-2}$ and $p<\frac{2d}{d-2}$. Harmonic analysis techniques seem most suited for smaller value of $p$, while number theory works well with the analysis of exponential sums experiencing constructive interference. See \cite{GermainMyersonPezzi} for an exploration of this dichotomy in the case of spectral projectors. There, sharp estimates are proven both using the circle method and cancellation between exponential sums.

\subsection{Notation}\label{1: Subsection Notation}
We use $|\cdot|$ to denote both the magnitude of a vector and the Lebesgue measure of a set; which interpretation is being used is always clear from context. We use $\#A$ to denote the cardinality of a set $A$ and to emphasize that the set is finite.

We denote $\mathbf{k}$ a vector in $\Z{d}$ with components $(k_1,...,k_d)$. The variable $x$ will denote a point in $\T{d}$. We often abuse notation to identify this with a point in the fundamental domain $[0,1]^d$. This is most commonly done in the $d=1$ case to identify $[0,1]$ and $S^1$.

We write $A\lesssim B$ to mean that $A\leq C B$ for some implicit constant $C$. If we wish to emphasize that $C  = C(P)$ for some parameter list $P$, we write $A\lesssim_P B$. We use similar notation for $A\gtrsim B$. We write $A\sim B$ to mean $A\lesssim B$ and $B\lesssim A$ (with the same parameter list). We will always allow the implicit constant to depend on the dimension and the value $p$, but never on the spectral parameter $N$.

We use $(a,b)$ to denote the greatest common divisor of integers $a$ and $b$.

We will write $e(y) = e^{2\pi i y}$. For a function defined on $\R{d},$ we will write $\hat{f}$ for the Fourier transform:

\begin{equation*}
    \hat{f}(\xi) = \int_{\R{d}}e(-\xi \cdot x)f(x)dx.
\end{equation*}
For a function $g$ defined on $\T{d}$, we will write its Fourier coefficients as $\mathcal{F}(g)(\mathbf{k})$. That is

\begin{equation*}
    g(x) = \sum_{\mathbf{k}\in\Z{d}}\mathcal{F}(g)(\mathbf{k})e(\mathbf{k}\cdot x).
\end{equation*}
The $\mathbf{k}$-th Fourier coefficient is 

\begin{equation*}
    \mathcal{F}(g)(\mathbf{k}) = \int_{\T{d}}e(-\mathbf{k}\cdot x)g(x)dx.
\end{equation*}
To ease comparison, we have tried where possible to emulate the notation of \cite{bourgainDemeterDiscreteRestrictImprovements2013}.

\subsection{Thanks} The author would like to thank Xiaoqi Huang and Connor Quinn for careful readings of this and earlier drafts of this manuscript. The author would also like to thank Simon Myerson for clarifying conversations on exponential sums which improved the presentation of Section 4, and Jonathan Hickman for pointing out an error in the definition of the local operator in Section 3.
\section{A $\log$ loss argument}\label{Section2}
\noindent
In this section we prove \Cref{1: logMainResultThm}.

\subsection{Preliminary estimates}
Let $K(x)$ be the convolution kernel (over $\T{d}$) of $P_N$. As $K(x)$ is the sum of exponentials whose frequencies are given by the lattice points on $NS^{d-1}$, it is difficult to analyze directly. It can be written as

\begin{equation}\label{2: kernel1stDefEqn}
    K(x) \coloneq \sum_{\mathbf{k}\in\mathcal{F}_{d,N}}e(\mathbf{k}\cdot x).
\end{equation}
It is beneficial to represent this discrete sum in a continuous way. The integration kernel $K(x)$ can be realized as

\begin{equation}\label{2: intKernelEquation}
    K(x) = \int_{[0,1]}\prod_{j=1}^d \bigPara{\sum_{k\in\Z{}}\gamma(k/N)e(kx_j+k^2t)} e(-\la t)dt.
\end{equation}
The function $\gamma$ is a smooth bump function supported in the ball of radius $\lesssim 1$. We note that everything in the integrand is $1$-periodic, and so the integration can be viewed as being over (a rescaled version of) $S^1$. The above representation selects out $\mathbf{k}\in\Z{d}$ such that $|\mathbf{k}|^2-\la = 0$, as in any other case the integral is $0$ by periodicity. Taking inspiration from the circle method from analytic number theory, we expect the integrand to be larger near rational times with small denominator. In this instance, it is productive to cut around times with very small denominator, as the rest of the function will have small Fourier coefficients.

Fix integer $Q$ such that $N\leq Q\leq N^2$. Define

$$D_{Q}\coloneq \{q: Q\leq q \leq 2Q,\, q\text{ is prime} \}.$$
We then consider 

$$R_{Q} \coloneq \{\frac{a}{q}: q\in D_Q, (a,q)=1\}.$$
It is immediate that $\#R_{Q}\sim Q^2(\log Q)^{-1}$. Let $1_{[-1/20,1/20]}\leq \eta \leq 1_{[-1/10,1/10]}$ be a Schwartz bump function. Define

\begin{equation}\label{2: etaQDef}
    \eta_Q = c_Q \sum_{a/q\in R_{Q}}\eta((t-\frac{a}{q})Q^2).
\end{equation}
The constant $c_Q$ is chosen so that $\int_{[0,1]} \eta_Q dt = 1$. Note that $c_Q \sim \log Q$ by direct computation. This will help get an improved bound for Fourier coefficients on the error term later. We now define

\begin{equation*}
    K^Q\coloneq \int_{[0,1]}\prod_{j=1}^d G(t,x_j)e(-\la t)\eta_Q(t)dt.
\end{equation*}
Here, $G(t,x_j) = \sum_{k\in\Z{}}\gamma(k/N)e(kx_j+k^2t)$ is the 1-dimensional Schr\"odinger propagator on the square torus. We now seek to bound the size of $K^{Q}$ and the size of the Fourier transform of $K - K^{Q}$. To achieve the first task, we need a bound on $G(t,x)$ (we drop the subscript $i$ momentarily). A sharp bound can be proven for all time using classical methods, however, we also require an estimate that applies near fractions of very large denominator. 

\begin{proposition}\label[proposition]{2: kernelBoundProp}
Let $q\geq N$ be an odd integer. Let $t= \frac{a}{q}+\varphi$ where $(a,q)=1$ and $|\varphi|\lesssim\frac{1}{q^2}$. Then

\begin{equation}\label{2: kernelBoundEqn}
    |G(t,x)|\lesssim q^{1/2}.
\end{equation}
    
\end{proposition}

Before we prove this, let us analyze $G(t,x)$. We will represent integers $k$ based on $q$, that is we shall write $k = rq+k_1$ for $0\leq k_1\leq q-1$. Poisson summation implies

    \begin{align*}
        G(t,x)& = \sum_{k_1=0}^{q-1}e(k_1^2a/q)\sum_{r\in\Z{}}\gamma(\frac{k_1+rq}{N})e((rq+k_1)x+(rq+k_1)^2\varphi)\\
        & = \sum_{m\in\Z{}}\bigPara{\frac{1}{q}\sum_{k_1=0}^{q-1}e(k_1^2a/q-k_1m/q)}\bigPara{\int_{\R{}}\gamma(y/N)e((x+\frac{m}{q})y+\varphi y^2)dy}\\
        &\coloneq \sum_{m\in\Z{}}S(a,m,q)J(x,\varphi,m,q).
    \end{align*}
We note that $S(a,m,q)$ is, up to the $q^{-1}$ factor, a generalized quadratic Gauss sum. This allows for the following estimate.

\begin{equation}\label{2: GaussSumBoundEqn}
    |S(a,m,q)|\lesssim q^{-1/2}.
\end{equation}

We defer the proof to \Cref{4: GaussSumBoundProp}. In this section we are restricted to $q$ prime, but in the following section we shall require bounds for all integers $q$. The above bounds holds for any integers $a,m$ and $q$ such that $(a,q)=1$.

We now return to the proof of the proposition.
\begin{proof}[Proof of \Cref{2: kernelBoundProp}]
   Fix $x\in\T{d}$. We have the following trivial estimate.

   \begin{equation*}
       |J(x,\varphi,m,q)|\lesssim N.
   \end{equation*}
   As we are considering times very close to a rational number, this estimate is superior to any van der Corput estimate. However, we will use (non) stationary phase to reduce the number of $m$ we sum over. Recall

   \begin{equation*}
       J(x,\varphi, m, q) =\int_{\R{}}\gamma(y/N)e((x+\frac{m}{q})y+\varphi y^2)dy.
   \end{equation*}
    Define $L(y) = 2\pi i[(x+\frac{m}{q})y+\varphi y^2]$. If $L'(y)=0$ for some $y$ in the domain, we will just use the trivial estimates. As $y$ is restricted to $B(0,cN)$, $|\varphi|\lesssim q^{-2}$, and $q\geq N$, we have $y\varphi\lesssim N^{-1}$. We now analyze the magnitude of $(x+\frac{m}{q})$. We have 

    \begin{equation*}
        |x+\frac{m}{q}|\lesssim N^{-1}
    \end{equation*}
    for $\lesssim q/N$ values of $m$. These will form the main contribution to our estimate. Now, if $|x+\frac{m}{q}|\sim 2^j N^{-1}$, this gives that

    \begin{equation*}
        |L'(y,N)|\sim 2^{j}N^{-1}.
    \end{equation*}
    Because of this, we may integrate by parts. Note

    \begin{equation*}
       J(x,\varphi, m, q) =\int_{\R{}}\gamma(y/N)\frac{1}{L'(y)}\frac{d}{dy}e^{L(y)}dy.
   \end{equation*}
   This gives

       \begin{equation*}
       J(x,\varphi, m, q) =-\int_{\R{}}\bigPara{\frac{\gamma'(y/N)}{NL'(y)}-\frac{\gamma(y/N)L''(y)}{(L'(y))^2}}e^{L(y)}dy.
   \end{equation*}
   The triangle inequality then gives

   $$|\frac{\gamma'(y/N)}{NL'(y)}-\frac{\gamma(y/N)L''(y)}{(L'(y))^2}|\lesssim 2^{-j}.$$
    This is because $|L''(y)|\lesssim N^{-2}$ and the estimates previously mentioned. Noting that $L^{(3)}(y)=0$, we may repeatedly integrate by parts, gaining a power of $2^{-j}$ each time. After $A$ iterations we get, for $m$ such that $|x+\frac{m}{q}|\sim 2^j N^{-1}$, that

    \begin{equation*}
        |J(x,m,\varphi,q)|\lesssim_A 2^{-jA}N.
    \end{equation*}
    There are $\lesssim 2^jqN^{-1}$ values of $m$ such that $|x+\frac{m}{q}|\sim 2^j N^{-1}$. This allows us to conclude by summing a geometric series and using \eqref{2: GaussSumBoundEqn}.

    \begin{align*}
        |\sum_{m\in\Z{}}S(a,m,q)J(x,\varphi,m,q)|&\lesssim q^{-1/2}\sum_{m\in\Z{}}|J(x,\varphi,m,q)| \\
        &\lesssim_A q^{-1/2}\bigPara{N\frac{q}{N}+\sum_{j}2^{-jA}2^{j}N\frac{q}{N}}.\\
        &\lesssim q^{1/2}.
    \end{align*}
    We have taken $A$ large enough in a manner not depending on $N$.
\end{proof}
\noindent
We now state two related lemmas that work for $q\leq N$. 

\begin{lemma}[Dirichlet's Lemma]\label{2: dirichlet'sLemma}
    For every $t\in[0,1]$, there exists a $2\leq q \leq N$ and $a$ with $(a,q)=1$ such that $t = \frac{a}{q}+\varphi$, where $|\varphi|\leq \frac{1}{Nq}$. If $t$ is near $0$ or $1$, we allow $a=0$ or $1$ respectively.
\end{lemma}

\begin{lemma}\label{2: propogatorDisperseiveLemma}
Let $1\leq a < q <N$ with $(a,q)=1$. For $t$ such that $|t- \frac{a}{q}|\leq \frac{1}{qN}$, we have that

\begin{equation}\label{2: propogatorDisperseiveEqn}
    |G(t,x)|\lesssim q^{-1/2}\min \{N, |t-\frac{a}{q}|^{-1/2}\}.
\end{equation}
\end{lemma}
See Lemma 3.18 in \cite{BourgainGAFALattice11993} for a proof of the second lemma, the first is classical. It is crucial that \eqref{2: kernelBoundEqn} and \eqref{2: propogatorDisperseiveEqn} do not have an $\eps$-loss, as we will be able to efficiently control the amount of time our function can be large. In particular, we can prove the following bound for $K^Q$ when $Q$ is very large. 

\begin{proposition}\label{2: KQSpaceBoundProp}
    Let $N\leq Q\leq N^2$ and $K^Q$ be as before. Then, for $d\geq 4$, the following $\log$ loss version of Proposition 4.2 in \cite{bourgainDemeterDiscreteRestrictImprovements2013} holds.

    \begin{equation}\label{2: KQSpaceBoundEqn}
        \lpn{K^Q}{\infty}\lesssim \frac{\log Q}{\log(NQ^{-1/2})} N^2 Q^{\frac{d-4}{2}}.
    \end{equation}
\end{proposition}

\begin{proof}
    Observing \eqref{2: propogatorDisperseiveEqn}, if $|G(t,x)|\gtrsim 2^s$, we immediately have 

    \begin{align*}
        &q\lesssim N^22^{-2s},\\
        &|t-\frac{a}{q}|\lesssim q^{-1}2^{-2s}.
    \end{align*}
    So the integrand of $K^Q$ can only be this large in certain neighborhoods of times that can be written as $\frac{a}{q}$ where $(a,q)=1$ and $q$ prime. A direct calculation shows

    \begin{equation*}
        |\{t\in[0,1]:|G(t,x)|\geq 2^s\}|\lesssim \sum_{q\lesssim N^2 2^{-2s}}\frac{\phi(q)}{q2^{2s}}\lesssim N^{2}2^{-4s}(\log(\frac{N}{2^s}))^{-1}.
    \end{equation*}
    Here, $\phi(q)$ is the Euler totient function, which counts the number of integers less than $q$ and coprime to it. The conclusion of \Cref{2: kernelBoundProp} says that, on the support of $\eta_Q$, we have

    \begin{equation*}
        |G(t,x)|\lesssim \sqrt{Q},
    \end{equation*}
    as we have localized to small neighborhoods are rationals with large denominators. We then immediately have via H\"older's inequality and our normalization constant $c_Q$ that

    \begin{align*}
        \lpnorm{G(t,x)}{d}{\supp \eta_Q}^d&\lesssim \log Q\,N^2\sum_{N^{d/2}\leq 2^s \lesssim \sqrt{Q}}(\log(\frac{N}{2^s}))^{-1}2^{s(d-4)} + \log QN^{\frac{d}{2}}\\
        &\lesssim \frac{\log Q}{\log(N Q^{-1/2})}\, N^2Q^{\frac{d-4}{2}}.
    \end{align*}
\end{proof}
It seems difficult to get around the $\log Q$ loss here as we will eventually take $Q= N^2$ in some regimes. The proof of the next proposition will illustrate another $\log Q$ loss. To continue, we prove that the Fourier transform of $K-K^Q$ is small.

\begin{proposition}\label{2: KerrFourierBoundProp}
    For any $d\geq 1$ and $N\leq Q\leq N^2$, we have

    \begin{equation}\label{2: KerrFourierBoundEqn}
        \lpn{\mathcal{F}(K-K^Q)}{\infty}\lesssim \log Q\,Q^{-1}
    \end{equation}
\end{proposition}

\begin{proof}
    We compute the Fourier coefficient of $K-K^Q$ as

    \begin{equation*}
        \mathcal{F}(K-K^Q)(\mathbf{k}) = \int_{\T{d}}\int_{[0,1]}\prod_{j=1}^dG(t,x_j)\bigPara{1-\eta_Q(t)} e(-\la t)e(-\mathbf{k}\cdot x)dtdx.
    \end{equation*}
    After exchanging the order of integration, we see that the integral with respect to $x$ is only non-zero when $\mathbf{k} = (k_1,...,k_d)$ where $k_j$ come from the sum in the definition of $G(t,x_j)$. This then leaves us with only a Fourier transform with respect to time, and so 

    \begin{equation*}
        \mathcal{F}(K-K^Q)(\mathbf{k}) = \widehat{1-\eta_Q}(|\mathbf{k}|^2-\la)\prod_{j=1}^d\gamma(\frac{k_j}{N}).
    \end{equation*}
    Define $l = |\mathbf{k}|^2-\la$. The bump functions do not alter our conclusions as they equal $1$ if $l = 0$. As $1-\eta_Q(t)$ has mean zero, we get

    \begin{equation*}
        \mathcal{F}(K-K^Q)(\mathbf{k}) = 0 \text{ if } |\mathbf{k}|^2=\la.
    \end{equation*}
    If $l\neq 0$, the constant term drops out and we compute the Fourier transform of $\eta_Q$. Recalling \eqref{2: etaQDef}, we have

    \begin{equation}\label{2: FourierTransNumberThrBoundEqn}
        \widehat{1-\eta_Q}(l) \sim c_Q Q^{-2}\hat{\eta}(\frac{l}{10Q^2})\sum_{q\in R_Q}\sum_{a=1}^{q-1}e(la/q).
    \end{equation}
    This follows from basic properties of the Fourier transform. Recall the $c_Q$ factor is $\sim \log Q$. As the sum is over roots of unity, the inner sum is $-1$ if $q$ does not divide $l$ and $q-1$ otherwise. Therefore, we have

    \begin{equation*}
        \widehat{1-\eta_Q}(l) \sim \log Q\, Q^{-2}\hat{\eta}(\frac{l}{10Q^2})\bigPara{-\#D_Q+\sum_{q\in D_Q,\, q|l}q}.
    \end{equation*}
    We have $\#D_Q\sim Q (\log Q)^{-1}$ and $q\sim Q$. Suppose $l$ has less than 10 factors in $D_Q$. Then the triangle equality immediately implies that

    \begin{equation*}
        |\widehat{1-\eta_Q}(l)| \sim \log Q\,Q^{-2}|\hat{\eta}(\frac{l}{10Q^2})|\cdot 10 Q \lesssim \log Q\,Q^{-1}.
    \end{equation*}
    Now suppose $l$ has $S$ factors with $S\geq 10$. Then $l\gtrsim Q^{S}$. However, $\eta$ is Schwartz and so enjoys rapid decay.

    \begin{equation*}
        |\hat{\eta}(z)|\lesssim (1+|z|)^{-100}.
    \end{equation*}
    Therefore $|\hat{\eta}(\frac{l}{10Q^2})|\lesssim Q^{-100\cdot(S-2)}$. The triangle inequality then gives

    \begin{equation*}
        |\widehat{1-\eta_Q}(l)| \lesssim \log Q\,Q^{-2}Q^{-100\cdot(S-2)} S Q\lesssim \log Q Q^{-1}, \, S\geq 10.
    \end{equation*}
    So in any case, we conclude.
\end{proof}

\subsection{Proof of $\log$-loss result}
We now prove \Cref{1: logMainResultThm}. Following \cite{bourgainDemeterDiscreteRestrictImprovements2013}, assume $\|a_\xi\|_{\ell^2(\mathcal{F}_{d,\la})}=1$. Define

$$F(x) = \sum_{\xi\in \mathcal{F}_{d,N}}a_\xi e(x\cdot \xi).$$
Recall the superlevel sets $E_\alpha$ for this function are given by

$$E_\alpha = \{x\in\T{d}: |F(x)|>\alpha\}.$$
This allows us to define the oscillatory indicator functions as

\begin{equation*}
    f(x)=\frac{F(x)}{|F(x)|}1_{E_{\alpha}}(x).
\end{equation*}
These functions have the following crucial property, which is trivially verified.

\begin{equation*}
    \int_{\T{d}}|f(x)|^pdx = |E_\alpha|.
\end{equation*}
By the definitions of these functions, Plancherel's Theorem, and Cauchy-Schwarz, we have

\begin{equation*}
    \alpha |E_\alpha|\leq \int_{\T{d}}\bar{F}(x)f(x)dx = \sum_{\xi\in \mathcal{F}_{d,N}}\bar{a}_\xi \mathcal{F}(f)(\xi),
\end{equation*}
which implies

\begin{equation*}
        \alpha^2 |E_\alpha|^2\leq \sum_{\xi\in \mathcal{F}_{d,N}}|\mathcal{F}(f)(\xi)|^2 = \langle K*f,f \rangle.
\end{equation*}
From Proposition 4.1 in \cite{bourgainDemeterDiscreteRestrictImprovements2013}, we also have that

\begin{equation}\label{2: importedKQBoundSmallHeightEqn}
    \lpn{K^Q}{\infty}\lesssim Q^{\frac{d-1}{2}+\eps}.
\end{equation}
Recall, we can split our kernel as

\begin{equation*}
    K = K^{Q}+(K-K^{Q}).
\end{equation*}
This gives

\begin{align*}
   | \langle K*f,f \rangle| & = |\langle K^Q*f,f \rangle  + \langle (K-K^Q)*f,f \rangle|\\
     & \leq \lpn{K^Q}{\infty}\langle \lpn{f}{1}, |f| \rangle  + \lpn{\widehat{K-K^Q}}{\infty}\langle |\hat{f}|,|\hat{f}| \rangle\\
     &\leq\lpn{K^Q}{\infty}|E_\alpha|^2  + \lpn{\widehat{K-K^Q}}{\infty}|E_\alpha|.
\end{align*}
We now optimize in $Q$ using \eqref{2: importedKQBoundSmallHeightEqn} or \Cref{2: KQSpaceBoundProp}, with the goal of making $\lpn{K^Q}{\infty}\leq \frac{1}{2}\alpha^2$. Having done this, we get

\begin{equation*}
    |E_\alpha|\lesssim \log Q\,\alpha^{-2}Q^{-1}.
\end{equation*}
Where $\alpha$ and $Q$ are now related by $\alpha^2\sim \log Q\,(\log (N/\sqrt{Q}))^{-1}\,N^2Q^{\frac{d-4}{2}}$ or $\alpha^2\sim Q^{\frac{d-1}{2}+\eps}$.

This yields for all $d\geq 5$

\begin{equation}\label{2: sharpLevelSetEstimatesEqun}
    |E_\alpha|\lesssim \begin{cases}
        \quad N^\epsilon \alpha^{-\frac{2(d+1)}{d-1}} &\text{ if }N^{\frac{d-1}{4}}\lesssim \alpha \lesssim N^{\frac{d-1}{3}+\eps_0}\\
        \quad (\log N)^{1+\frac{2}{d-4}} N^{\frac{4}{d-4}}\alpha^{-\frac{2d-4}{d-4}} &\text{ if }\alpha\gtrsim N^{\frac{d-1}{3}+\eps_0}
    \end{cases}.
\end{equation}
Here, $\eps_0>0$ is arbitrary. The bound for large heights can be improved to just a $\log N$ factor if $\alpha\lesssim N^{\frac{d-2}{2}-\eps_0}$. Note that the bound \eqref{2: KQSpaceBoundEqn} is vacuous if $d\leq 4$. The bound $|E_\alpha|\lesssim N^\epsilon \alpha^{-\frac{2(d+1)}{d-1}}$ actually holds for all level sets as a consequence of the conjecture being proven by decoupling when $p\leq 2(d+1)/(d-1)$. We can thus view this result as improving the level set bound for larger heights.

It is a standard fact that

\begin{equation*}
    \lpnorm{F}{p}{\T{d}}^p = \int_0^{\lpn{F}{\infty}}\alpha^{p-1}|E_\alpha|d\alpha.
\end{equation*}
The torus being the domain of integration does not factor into the above equality. We use \eqref{2: sharpLevelSetEstimatesEqun} and compute

\begin{align*}
    \int_0^{N^{\frac{d-1}{3}}}\alpha^{p-1}|E_\alpha|d\alpha&
    \lesssim N^\epsilon\int_0^{N^{\frac{d-1}{3}}}\alpha^{p-1}\alpha^{-\frac{2(d+1)}{d-1}}d\alpha\\
    &\lesssim N^\eps\alpha^{p-\frac{2(d+1)}{d-1}}|_0^{N^{\frac{d-1}{3}}}\\
    &\lesssim N^\epsilon N^{\frac{d-1}{3}(p-\frac{2(d+1)}{d-1})}.
\end{align*}
Which is $\lesssim N^{\frac{p(d-2)}{2}-d}$ when $p>\frac{2d}{d-3}$ and $d\geq 6$ or $p>6$ when $d=5$. We use our second bound for the large heights. If $d\geq 5$, we have

\begin{align*}
    \int_{N^{\frac{d-1}{3}}}^{N^{\frac{d-2}{2}}}\alpha^{p-1}|E_\alpha|d\alpha&
    \lesssim (\log N)^{1+\frac{2}{d-4}}\,N^{\frac{4}{d-4}}\int_{N^{\frac{d-1}{3}}}^{N^{\frac{d-2}{2}}}\alpha^{p-1}\alpha^{-\frac{2d-4}{d-4}}d\alpha\\
    &\lesssim (\log N)^{1+\frac{2}{d-4}}\,N^{\frac{4}{d-4}}\alpha^{p-\frac{2d-4}{d-4}}|_{N^{\frac{d-1}{3}}}^{N^{\frac{d-2}{2}}}\\
    &\lesssim (\log N)^{1+\frac{2}{d-4}}\,N^{\frac{4}{d-4}} N^{\frac{d-2}{2}(p-\frac{2d-4}{d-4})}\\
    &\sim (\log N)^{1+\frac{2}{d-4}}\,N^{\frac{p(d-2)}{2}-d}.
\end{align*}
And so we are done.
\section{A sharp argument}\label{Section3}
\noindent
In this section, we prove \Cref{1: SharpMainResultThm}.

\subsection{Preliminary estimates}
The convolution kernel (on $\T{d}$) of $P_N$ is given by

\begin{equation}\label{3: kernel1stDefEqn}
    K(x) \coloneq \sum_{\mathbf{k}\in\mathcal{F}_{d,N}}e(\mathbf{k}\cdot x).
\end{equation}
Taking inspiration from the circle method from analytic number theory, we replace this discrete sum with a more continuous equivalent. It is a straightforward consequence of periodicity that

\begin{equation}\label{3: kernelDefEqn}
    K(x) = \int_{[0,1]}\prod_{j=1}^dG(t,x_j) e(-\la t)dt,
\end{equation}
where

\begin{equation}\label{3: SchroPropG(t,x)DefEqn}
    G(t,x_j) = \sum_{k\in\Z{}}\gamma(k/N)e(kx_j+k^2t).
\end{equation}
The function $\gamma$ is a Schwartz bump function such that $ 1_{[-1,1]}\leq \gamma$. Because of the cancellation occurring in the time integral, the only essential property of this function is that it is 1 when $k\sim N$. Importantly, $G(t,x_j)$ in its given form is the Schr\"odinger propagator on the circle. Additionally, we view the integral in \eqref{3: kernelDefEqn} as being over (a rescaled version of) $S^1$ which has been identified with $[0,1]$ in the obvious way.

The motivating insight from the circle method is that any expression of the form \eqref{3: kernelDefEqn} should be large when $t$ is near a rational with small denominator. We shall quantify this intuition and use it to obtain bounds on $K(x)$, which we shall then convert to $L^p$ norms of eigenfunctions.

In this section we use different definitions for parameters as compared to \Cref{Section2}. Let $2\leq Q <N/100$ be dyadic. We also define the dyadic $2^s$ by $Q<2^s <N$. Let $\eta$ be a smooth bump function identically $1$ on $\frac{1}{4}\leq |t|\leq \frac{1}{2}$ and supported in $\frac{1}{8}\leq |t|\leq 1$. We define

$$R_Q = \{\frac{a}{q}: Q\leq q< 2Q, (a,q)=1\},$$
and

\begin{equation}\label{3: etaQDef}
    \eta_{Q,s}(t) = \sum_{a/q\in R_{Q}}\eta((t-a/q)N2^s).
\end{equation}
Note we place no restriction on $q$ other then it must be $\sim Q$. Clearly, $\#R_Q\sim Q^2$. Our definition of $\eta$ implies

$$\supp \eta_{Q,s} = \{t:|t-\frac{a}{q}|\sim \frac{1}{N2^s}, \frac{a}{q}\in R_Q\}.$$
When $2^s\sim N$, we modify the bump function $\eta$ so that $\eta_{Q,s}$ is supported on $|t-\frac{a}{q}|\lesssim N^{-2}$. This will not the affect bulk of our analysis. It is trivially verified that the supports of $\eta_{Q,s}$ do not overlap. 

We now split the kernel into three pieces. For the first, let $\eta_0$ be an appropriate compactly supported bump function that is $1$ near $0$. We define

\begin{align*}
    &K_0 \coloneq \int_{[0,1]}\prod_{j=1}^d G(t,x_j)e(-\la t)\eta_0(N^2 t)dt,\\
    &K_{Q,s} \coloneq \int_{[0,1]}\prod_{j=1}^d G(t,x_j)e(-\la t)\eta_{Q,s}(t)dt,\\
    &K_{err} \coloneq K - K_0 - \sum_{Q,s}K_{Q,s}.
\end{align*}
The first piece of the kernel is supported only around $t=0$. We could alter the lower bound on $Q$ to include this portion of the domain in the definition of the $K_{Q,s}$, but we will present a simpler and distinct argument for $K_0$. We recall that we identify $[0,1]$ with $S^1$ to make sense of this definition and the dyadic arcs around $0$. The cutoff function $\eta_{Q,s}$ is supported in time near rationals with denominator $\sim Q$ at a distance $\sim (N2^s)^{-1}$. The error $K_{err}$ is defined by a cutoff supported near times that cannot be well approximated by rationals of denominator $\lesssim N$. 

We recall two key inputs from number theory, Dirichlet's Lemma and the sharp bound for the Schr\"odinger propagator for the torus.

\begin{lemma}[Dirichlet's Lemma]\label{3: dirichlet'sLemma}
    For every $t\in[0,1]$, there exists a $2\leq q \leq N$ and $a$ with $(a,q)=1$ such that $t = \frac{a}{q}+\varphi$, where $|\varphi|\leq \frac{1}{Nq}$. If $t$ is near $0$ or $1$, we allow $a=0$ or $1$ respectively.
\end{lemma}

\begin{lemma}\label{3: propogatorDisperseiveLemma}
Let $1\leq a < q <N$ with $(a,q)=1$. For $t$ such that $|t- \frac{a}{q}|\leq \frac{1}{qN}$ we have that

\begin{equation}\label{3: propogatorDisperseiveEqn}
    |G(t,x)|\lesssim q^{-1/2}\min \{N, |t-\frac{a}{q}|^{-1/2}\}.
\end{equation}
\end{lemma}
See Lemma 3.18 in \cite{BourgainGAFALattice11993} for a proof of the second lemma, the first is classical. The estimate for $K_0$ is an instance of a more general phenomena for local operators.

\begin{proposition}\label{3: K0localBoundProp}
Let $d\geq 5$ and $2\leq p \leq \infty$. Let $K_0$ be as before. Then, for all $F\in L^{p'}$, the H\"older dual of $L^p$, we have

\begin{equation}\label{3: K0localBoundEqn}
    \lpnorm{K_0* F}{p}{\T{d}}\lesssim N^{d-2-\frac{2d}{p}}\lpnorm{F}{p'}{\T{d}}.
\end{equation}
\end{proposition}

\begin{proof}
    First we write

    \begin{equation*}
        K_0(x) = \int_{[0,1]}\eta_0(N^2t)\prod_{j=1}^d\sum_{k_j}\gamma(k_j/N)e(x_jk_j+k_j^2t)e(-\la t)dt.
    \end{equation*}
    We may take the Fourier transform with respect to time as, after some trivial identifications, we can replace $[0,1]$ with integration over $\R{}$. Then

    \begin{equation*}
        K_0(x) = \sum_{\mathbf{k}\in\Z{d}}N^{-2}\widehat{\eta_0}(N^{-2}(|\mathbf{k}|^2-\la))e(\mathbf{k}\cdot x).
    \end{equation*}
    Recall that $\eta_0$ is Schwartz, and so its Fourier transform enjoys rapid decay outside of $B(0,c)$ for some $c$. By the lattice point counts, we bound the number of $\mathbf{k}$ such that $|\mathbf{k}|^2-\la=X$ with $X\lesssim N^2$ by $N^2\cdot N^{d-2}$, where there are no additional $\eps$ factors as $d\geq 5$. By the triangle inequality, we therefore have

    \begin{equation}\label{3: K0LintyBoundInProofEqn}
        \lpn{K_0}{\infty}\lesssim N^{d-2}.
    \end{equation}
    We also bound the Fourier coefficients. To compute

    \begin{equation*}
        \mathcal{F}(K_0)(\mathbf{k}) = \int_{\R{d}}\bigPara{\int_{[0,1]}\eta_0(N^2t)\prod_{j=1}^d\sum_{k}\gamma(\frac{k_j}{N})e(x_jk_j+k_j^2t-\la t)dt}e(-\mathbf{k}\cdot x)dx,
    \end{equation*}
    we simply note that the only expressions which depends on $x$ inside the $t$ integral are the exponentials $e(x_jk_j)$. Using the cancellation in $x$ and taking the Fourier transform with respect to time yields

    \begin{equation*}
        \mathcal{F}(K_0)(\mathbf{k}) = N^{-2}\widehat{\eta_0}(N^{-2}(|\mathbf{k^2}|-\la))\prod_{j=1}^d\gamma(\frac{k_i}{N}).
    \end{equation*}
    Here, $\mathbf{k} = (k_1,...,k_d)$. The triangle inequality then implies

    \begin{equation}\label{3: FTK0LinftyBoundProofEqn}
        \lpn{\mathcal{F}(K_0)}{\infty}\lesssim N^{-2}.
    \end{equation}
    The bounds \eqref{3: K0LintyBoundInProofEqn} and \eqref{3: FTK0LinftyBoundProofEqn} immediately imply, for any $F$, the bounds

    \begin{align*}
        &\lpn{K_0*F}{2}\lesssim N^{-2}\lpn{F}{2},\\
        &\lpn{K_0*F}{\infty}\lesssim N^{d-2}\lpn{F}{1}.
    \end{align*}
    The result follows from interpolation.
\end{proof}
\begin{remark}\label{3: LossRemark}
    Our proof in dimension $2,3,4$ will also imply the sharp bound for $K_0*F$, but with an additional $\eps$ factor. Recall such a factor is necessary in $L^p$ estimates because the lower bound for the number of lattice points on spheres in low dimensions has such a factor. The advantage of isolating the local contribution is that we are now comfortable proving stronger bounds for the other pieces of the kernel without fear of a contradiction, although this does not play into our proof.
\end{remark}
\noindent
The main difficultly will be in estimating $K_{Q,s}$. We have the following.

\begin{proposition}\label{3: KQsLinftyBoundProp}
    For each dyadic $2\leq Q<\frac{1}{100}N$ and $Q<2^s<N$, we have

    \begin{equation}\label{3: KQsLinftyBoundEqn}
        \lpnorm{K_{Q,s}}{\infty}{\T{d}}\lesssim (N2^s)^{\frac{d}{2}-1}Q^{-\frac{d-4}{2}}.        
    \end{equation}
\end{proposition}
Using the bound in \Cref{3: propogatorDisperseiveLemma} and H\"older's inequality will prove the result almost immediately. We provide an ostensibly more complicated proof where we write out $K_{Q,s}(x)$ in full using the formulation reached after the statement of \Cref{2: KQSpaceBoundProp}. This is for the purposes of illustrating the difficulty in improving \eqref{3: KQsLinftyBoundEqn}. This gives an expression for $K_{Q,s}(x)$ as

\begin{equation}\label{3: KQsFullyWrittenOutEqn}
\int_{[0,1]}\sum_{\mathbf{m}\in\Z{d}}\bigPara{\prod_{j=1}^dS(a,m_j,q)e(-\la \frac{a}{q})}\bigPara{\prod_{j=1}^dJ(x_j,\varphi,m_j,q)e(-\la\varphi)}\eta_{Q,s}(t)dt.
\end{equation}
This obviously unwieldy expression can be simplified. Recall that $\eta_{Q,s}(t)$ localizes to $\sim (N2^s)^{-1}$ neighborhoods of rational $a/q$ such that $q\sim Q$ and $(a,q)=1$. The first factor in the sum is independent of $\varphi$; it only depends on $a$ and $q$. Because of this, let us define 

\begin{equation}
    \mathcal{S}(m,q) \coloneq \sum_{(a,q)=1}\prod_{j=1}^dS(a,m_j,q)e(-\la \frac{a}{q}).
\end{equation}

This term will appear naturally when we integrate, while the term involving $J$ will be dealt with by similar techniques to what we used in \Cref{Section2}.

By the triangle inequality we have

\begin{equation}\label{3: S(m,q)TriangleIneqEqn}
    |\mathcal{S}(m,q)|\lesssim q^{1-\frac{d}{2}}.
\end{equation}

Now we prove \Cref{3: KQsLinftyBoundProp}.
\begin{proof}
Momentarily drop the index $j$. Recall

\begin{equation*}
    J(x,\varphi, m, q) =\int_{\R{}}\gamma(y/N)e((x+\frac{m}{q})y+\varphi y^2)dy.
\end{equation*}
In contrast to the previous section, $q\leq N$ here. The trivial inequality and stationary phase imply as before that

\begin{equation*}
    |J(x,\varphi, m, q)|\lesssim \min \{N,|\varphi|^{-1/2}\}.
\end{equation*}
However, as we are dealing with smaller $q$, only $\sim 1$ of the $m$ will significantly contribute. Recall that if

\begin{equation*}
    |x+\frac{m}{q}|\lesssim N^{-1},
\end{equation*}
we use the trivial estimate, but by repeated integration by parts we have that

\begin{equation*}
    |x+\frac{m}{q}|\sim 2^j N^{-1}
\end{equation*}
implies

\begin{equation*}
    |J(x,\varphi, m, q)|\lesssim_A 2^{-jA} \min \{N,|\varphi|^{-1/2}\}.
\end{equation*}
There are $\sim 2^j$ such $m$, and only $\sim 1$ integers $m$ such that $|x+\frac{m}{q}|\lesssim N^{-1}$. By integrating by parts a finite number of times not depending on $N$, we get enough decay in $2^j$ that we may sum the series. Because of this we have

\begin{equation*}
    \sum_{\mathbf{m}\in\Z{d}}\bigl|\prod_{j=1}^dJ(x_j,\varphi,m_j,q)e(-\la\varphi)\bigr|\lesssim \min\{N^d,|\varphi|^{-d/2}\}.
\end{equation*}
We need some care before we deploy this bound. Recalling \eqref{3: KQsFullyWrittenOutEqn}, we will first break the domain up based on the disjoint components of $\eta_{Q,s}$, which are centered around fractions $a/q$. On each of these components, $S(a,m,q)$ is constant as it does not depend on $\varphi$, and so we can sum in $a$ and $q$. Additionally, $J(x,\varphi,m,q)$ does not depend on $a$ and so has the same values on any component of the support of $\eta_{Q,s}$. This gives after a change of variables

\begin{equation*}
\sum_{q\sim Q}\sum_{\mathbf{m}\in\Z{d}}\mathcal{S}(m,q)\int_{E}\bigPara{\prod_{j=1}^dJ(x_j,\varphi,m_j,q)e(-\la\varphi)}\eta_{Q,s}(t)dt.
\end{equation*}
Here, $E$ is a set of measure $\sim (N2^s)^{-1}$. As we have absorbed the sum in $a$, we now use the triangle inequality. Our bound for $\mathcal{S}(m,q)$ is independent of $m$, so, much like in the proof of \Cref{2: propogatorDisperseiveLemma}, we focus on the $J$ term to sum. 

Over the domain we are integrating in the expression for $K_{Q,s}$, we have that $|\varphi|\sim (N2^s)^{-1}$. So this bound becomes

\begin{equation*}
    \sum_{\mathbf{m}\in\Z{d}}\bigl|\prod_{j=1}^dJ(x_j,\varphi,m_j,q)e(-\la\varphi)\bigr|\lesssim (N2^s)^{d/2}.
\end{equation*}
When $2^s\sim N$ we use the trivial estimate. Returning to \eqref{3: KQsFullyWrittenOutEqn}, we see the support of $\eta_{Q,s}$ means we are integrating over a set of measure $(N2^s)^{-1}$, but with a sum over $a$ and $q$. The sum over $a$ is absorbed into the definition of $\mathcal{S}(m,q)$, and so by the triangle inequality and \eqref{3: S(m,q)TriangleIneqEqn} we have

\begin{equation*}
    |K_{Q,s}(x)|\lesssim (N2^s)^{d/2-1}\sum_{q\sim Q}q^{1-\frac{d}{2}}.
\end{equation*}
We conclude by summing.

\begin{equation*}
    |K_{Q,s}(x)|\lesssim (N2^s)^{d/2-1}Q^{-\frac{d-4}{2}}.
\end{equation*}
\end{proof}
\noindent
Now we import a bound for the error term.

\begin{proposition}\label{3: KerrSpaceBoundProp}
    We have

    \begin{equation}\label{3: KerrSpaceBoundEqn}
        \lpn{K_{err}}{\infty}\lesssim N^{\frac{d-1}{2}+\eps}.
    \end{equation}
\end{proposition}
This is a direct consequence of (2.15) in \cite{bourgainTorusEFunc1993}, taking $Q\sim 2^s\sim N$. Note that \Cref{3: KQsLinftyBoundProp} would imply $\lpn{K_{err}}{\infty}\lesssim N^{\frac{d}{2}}$. Such improvements are not obvious for \eqref{3: KQsLinftyBoundEqn} using cancellation like is done in cited works because of our intolerance of $\eps$-losses. This will be discussed more a the end of Section 4, but the improved estimates in \cite{bourgainDemeterDiscreteRestrictImprovements2013} only hold for $Q\geq N^{c}$ for small $c>0$. As the error term is essentially the case $Q\sim N$, there is no harm having an $\eps$-loss in $N$ or $2^s$ there. The main term in $Q\sim 1$. Our pursuit of sharp estimates often limits us to only using the trivial bounds for exponential sums; finding a way around this would improve the results presented here. A similar restriction appears in the proof of the next proposition.

It will also be beneficial to compute the magnitude of the Fourier transform of $K_{Q,s}$. Such an estimate can be used in two ways, for proving $L^2\rightarrow L^2$ bounds or for a specific argument involving level sets. In this section we will use the latter. In the following estimates, it is important again that no $\eps$-losses appear in the parameters $N$ and $2^s$.

\begin{proposition}\label{3: KQsFourierCoeffBoundProp}
    For any $Q,s,$ and $\mathbf{k}\in\Z{d}$, we have

    \begin{equation}\label{3: smallQKQsFourierCoeffBoundEqn}
        |\mathcal{F}(K_{Q,s})(\mathbf{k})| \lesssim \frac{Q^2}{N2^s}.
    \end{equation}
\end{proposition}

\begin{proof}
    When computing $\mathcal{F}(K_{Q,s})(\mathbf{k})$, we note that 

    $$G(t,x) = \sum_{k\in \Z{}}\gamma(k/N)e(kx+k^2t^2)$$
    is the only $x$ dependence in the definition of $K_{Q,s}$. So, like in the proof of \Cref{3: K0localBoundProp}, we may directly take the Fourier transform in $x$ and then in time, using the small support of $\eta_{Q,s}$ to view the transform in time as occurring on $\R{}$. Along with the product structure on the torus, this gives that

    \begin{equation*}
        \mathcal{F}(K_{Q,s})(\mathbf{k}) = \widehat{\eta_{Q,s}}(|\mathbf{k}|^2-\la)\prod_{i=1}^d\gamma(k_i/N).
    \end{equation*}
    The smooth cutoffs play no role in our analysis going forward. Let $l= |\mathbf{k}|^2-\la$. Using linearity and the definition \eqref{3: etaQDef}, we compute $\widehat{\eta_{Q,s}}$ as

    \begin{equation*}
        \widehat{\eta_{Q,s}}(l) = (N2^s)^{-1}\hat{\eta}(\frac{l}{N2^s})\sum_{\frac{a}{q}\in R_Q}e(la/q).
    \end{equation*}
    We conclude by $\#R_Q\sim Q^2$ and the triangle inequality. 

    $$|\mathcal{F}(K_{Q,s})(\mathbf{k})|\lesssim (N2^s)^{-1}Q^2.$$
\end{proof}
\noindent
These are all the estimates required for our proof.
\subsection{Proof of sharp result}
Using the estimates proven in the previous subsection and basic inequalities, we have, for general $F$, that

\begin{align*}
    &\lpn{F* K_{Q,s}}{2}\lesssim \frac{Q^2}{N2^s}\lpn{F}{2},\\
    &\lpn{F* K_{Q,s}}{\infty}\lesssim (N2^s)^{\frac{d}{2}-1}Q^{-\frac{d-4}{2}}\lpn{F}{1}.
\end{align*}
We wish to interpolate these these two estimates and sum over $Q$ and $s$. The estimate will always be largest when $2^s\sim N$, and we will be able to proceed when the power of $Q$ is negative, which means the estimate is dominated by the $Q\sim 1$ terms. Importantly, we have the correct powers of $N$ and $2^s$ without any $\eps$-loss which, in the end, give us the desired sharp estimate that agrees with the local term $K_0$, albeit in a restricted range of $p$. Interpolation yields

\begin{equation}\label{3: balanceQEqn}
    \lpn{F* K_{Q, s}}{p}\lesssim (N2^s)^{\frac{d-2}{2}-\frac{d}{p}}Q^{\frac{d}{p}-\frac{d-4}{2}}\lpn{F}{p'}.
\end{equation}
The power of $Q$ is negative precisely when $p>\frac{2d}{d-4}$. This yields

\begin{equation*}
    \lpn{F* \sum_{Q,s} K_{Q,s}}{p}\lesssim N^{d-2-\frac{2d}{p}}\lpn{F}{p'},\quad p>\frac{2d}{d-4}.
\end{equation*}
Recall that this is the same bound as for $K_0$. For ease of notation define $K_{main} = K_0 + \sum_{Q,s}K_{Q, s}$. The above bound therefore holds for $K_{main}$ when $p>\frac{2d}{d-4}$ by the triangle inequality. We are now in position to prove the result.

\begin{remark}
    Even with a bound that was finite for $d=4$, we would not prove corresponding sharp bounds because of the $\eps$-loss present in the lattice point count. This manifests both in the local estimate and in the $L^\infty$ bound used in the computation below.
\end{remark}
\noindent
Following \cite{bourgainDemeterDiscreteRestrictImprovements2013}, assume $\|a_\xi\|_{\ell^2(\mathcal{F}_{d,\la})}=1$. Define

$$F(x) = \sum_{\xi\in \mathcal{F}_{d,N}}a_\xi e(x\cdot \xi).$$
Recall that the level sets $E_\alpha$ for this function are given by

$$E_\alpha = \{x\in\T{d}: |F(x)|>\alpha\}.$$
This allows us to define the oscillatory indicator functions as

\begin{equation*}
    f(x)=\frac{F(x)}{|F(x)|}1_{E_{\alpha}}(x).
\end{equation*}
These functions have the following crucial property, which is trivially verified.

\begin{equation*}
    \int_{\T{d}}|f(x)|^pdx = |E_\alpha|.
\end{equation*}
By the definitions of these functions, Plancherel's Theorem, and Cauchy-Schwarz, we have

\begin{equation*}
    \alpha |E_\alpha|\leq \int_{\T{d}}\bar{F}(x)f(x)dx = \sum_{\xi\in \mathcal{F}_{d,N}}\bar{a}_\xi \mathcal{F}(f)(\xi),
\end{equation*}
which implies

\begin{equation*}
        \alpha^2 |E_\alpha|^2\leq \sum_{\xi\in \mathcal{F}_{d,N}}|\mathcal{F}(f)(\xi)|^2 = \langle K*f,f \rangle.
\end{equation*}
Recall, we can split our kernel as

\begin{equation*}
    K = K_{main}+K_{err}.
\end{equation*}
For $p>\frac{2d}{d-4}$, we have the following estimate.

\begin{align*}
    \alpha^2|E_\alpha|^2&\leq \langle K * f, f \rangle = \langle K_{main} * f, f \rangle + \langle K_{err} * f, f \rangle \\
    &\leq \lpn{f}{p}\lpn{K_{main}*f}{p'} + N^{\frac{d-1}{2}+\eps}\langle|f|,|f| \rangle\\
    &\leq N^{d-2-\frac{2d}{p}}|E_\alpha|^{\frac{2}{p'}}+ N^{\frac{d-1}{2}+\eps}|E_\alpha|^2.
\end{align*}
If $\alpha^2\gtrsim N^{\frac{d-1}{2}+\eps_0}$, then the rightmost term can be absorbed into the left hand side. We then simplify to

\begin{equation}\label{3: sharpLevelSetBoundEqn}
    |E_\alpha| \leq N^{\frac{p(d-2)}{2}-d}\alpha^{-p},\quad p>2d/(d-4).
\end{equation}
The result now follows from standard estimates. Compute

\begin{equation*}
    \int_{\T{d}} |F|^p = \int_{N^{\frac{d-1}{4}+\eps_0}}^{N^{\frac{d-2}{2}}}\alpha^{p-1}|E_\alpha|d\alpha + \int_{0}^{N^{\frac{d-1}{4}+\eps_0}}\alpha^{p-1}|E_\alpha|d\alpha.
\end{equation*}
The second term is handled by the estimate $|E_\alpha|\lesssim N^\eps\alpha^{-\frac{2(n+1)}{n-1}}$, which is a corollary of the $\ell^2$-decoupling theorem. For the first term, we use \eqref{3: sharpLevelSetBoundEqn}. Let $p>p_1>2d/(d-4)$. Then

\begin{align*}
    \int_{N^{\frac{d-1}{4}+\eps_0}}^{N^{\frac{d-2}{2}}}\alpha^{p-1}|E_\alpha|d\alpha&\lesssim  N^{\frac{p_1(d-2)}{2}-d}\int_{N^{\frac{d-1}{4}+\eps_0}}^{N^{\frac{d-2}{2}}}\alpha^{p-1}\alpha^{-p_1}d\alpha\\
    &\lesssim  N^{\frac{p_1(d-2)}{2}-d} \alpha^{p-p_1}\rvert_{N^{\frac{d-1}{4}+\eps_0}}^{N^{\frac{d-2}{2}}}\\
    &\lesssim   N^{\frac{p_1(d-2)}{2}-d}\cdot N^{\frac{d-2}{2}(p-p_1)}\\
    &\lesssim N^{\frac{p(d-2)}{2}-d}.
\end{align*}
The result follows.
\section{Exponential Sum Estimates}\label{SectionExpSumEstimates}
In the previous arguments, two exponential sums arose. In this section, we sketch the proof of the required bounds for the generalized quadratic Gauss sums and explore why the bounds used for Kloosterman and Sail\'e sums in previous works are insufficient for our purposes. 

Recall that the generalized quadratic Gauss sums $S(a,m,q)$ are given, for integers $a,m,q$ such that $(a,q)=1$, by

\begin{equation}\label{4: S(a,m,q)DefEqn}
    S(a,m,q) = \frac{1}{q}\sum_{k=0}^{q-1}e(\frac{ak^2+mk}{q}).  
\end{equation}
We use a different sign convention for $m$ in this section, but this has no bearing on our results. We also have the following expression, which plays the role of the singular series from the circle method.
\begin{equation}\label{4: FancyS(m,q)DefEqn}
    \mathcal{S}(m,q) = \sum_{(a,q)=1}(\prod_{j=1}^dS(a,m_j,s)e(-\la \frac{a}{q})). 
\end{equation}

\subsection{Gauss Sum Estimates}\label{4: SubsectionGaussSumEstimates}
For the generalized quadratic Gauss sum, we wish to prove the following.

\begin{proposition}\label{4: GaussSumBoundProp}
    For $q\geq 2$, we have 

    \begin{equation}\label{4: GaussSumBoundEqn}
        |S(a,m,q)|\lesssim q^{-1/2}.
    \end{equation}
\end{proposition}
We may further assume that $(a,q)=1$, however the multiplicative nature of quadratic Gauss sums means this reduction is redundant for the bound we are trying to state. We will show that these sums exhibit square root cancellation. The $S(a,m,q)$ generalize the standard quadratic Gauss sums.

\begin{equation*}
    S(1,0,q) = \frac{1}{q}\sum_{k=0}^{q-1}e(\frac{k^2}{q}).
\end{equation*}
It is known this sum exhibits square root cancellation, which we prove here for odd primes.

\begin{lemma}[Standard Gauss Sum Bound]\label{4: stdGaussSumBndLemma}
For any odd prime $p$, we have

\begin{equation*}\label{4: stdGaussSumBndEqn}
    |S(1,0,p)| \leq p^{-1/2}.
\end{equation*}
    
\end{lemma}

\begin{proof}
We have

\begin{align*}
    p^2|S(1,0,p)|^2 &= \sum_{k=1}^{p-1}e(k^2/p)\overline{\sum_{l=1}^{p-1}e(l^2/p)} \\
    & = \sum_{k=0}^{p-1}\sum_{l=0}^{p-1}e((k^2-l^2)/p)\\
    & = \sum_{k=0}^{p-1}\sum_{l=0}^{p-1}e((k+l)(k-l))/p).
\end{align*}
We now make the change of variables $\alpha = k+l$ and $\beta=k-l$. This gives

\begin{equation*}
    p^2|S(1,0,p)|^2 = \sum_{\alpha=0}^{p-1}\sum_{\beta=0}^{p-1}e(\frac{4\alpha \beta}{p}).
\end{equation*}
The assumption that $p$ is odd means that multiplication by $4$ as we sum over the cyclic groups is an isomorphism, and so we can replace it with $1$. If $\alpha \neq 0$, then, by the geometric series formula, the sum in $\beta$ is $0$. Otherwise, it is $p$. This gives

\begin{equation*}
    p^2|S(1,0,p)|^2 = p.
\end{equation*}
Which obviously allows us to conclude.
\end{proof}
To prove an identical bound for any odd natural number $q$, one may use the imported \Cref{4: QuadRecipGaussSumLemma} below. For a more elementary proof, we refer to \cite{gaussSumMurtySiddhi2017}. The case $p=2$ is trivial.  For general even numbers, there may be an additional factor of $\sqrt{2}$ if $q$ is a multiple of $4$. See Corollary 1.2.3 in \cite{GaussJacobiBook} for a full statement. With this imported estimate, we have \Cref{4: stdGaussSumBndLemma} for all integers after proving the next lemma.

We now also state a series of technical results which we will use, along with the more general version of \Cref{4: stdGaussSumBndLemma}, to prove \Cref{4: GaussSumBoundProp}.

\begin{lemma}[Gauss sums are multiplicative]\label{4: GaussSummsMultiplicativeLemma}
    Assume $q = p_1p_2$ with $(p_1,p_2)=1$. Then

    \begin{equation}\label{4: GaussSummsMultiplicativeEqn}
        S(a,m,q) = S(ap_2,m,p_1)S(ap_1,m,p_2).
    \end{equation}
\end{lemma}
\begin{proof}
    We compute directly

    \begin{align*}
         S(ap_2,m,p_1)S(ap_1,m,p_2)& = \sum_{k=0}^{p_1-1}e(\frac{ap_2k^2+mk}{p_1})\sum_{l=0}^{p_2-1}e(\frac{ap_1l^2+ml}{p_2}) \\ 
         & = \sum_{k=0}^{p_1-1}\sum_{l=0}^{p_2-1}e(\frac{ap_2^2k^2+p_2mk+ap_1^2l^2+p_1ml}{p_1p_2})\\
         &=\sum_{k=0}^{p_1-1}\sum_{l=0}^{p_2-1}e(\frac{a(p_2k+p_1l)^2+m(p_2k+p_1l)}{p_1p_2})
    \end{align*}
    Because of our assumption that $a,p_1,$ and $p_2$ are all coprime, the Chinese Remainder Theorem implies that $p_2k+p_1l$ runs through all integers mod $p_1p_2$ as $k$ and $l$ vary. This enables the change of variables $z=p_2k+p_1l$, which yields

    \begin{equation*}
        S(ap_2,m,p_1)S(ap_1,m,p_2) = \sum_{z=0}^{p_1p_2-1}e(\frac{az^2+mz}{p_1p_2}) = S(a,m,q).
    \end{equation*}
\end{proof}

\begin{lemma}\label{4: OddForPower4Lemma}
    Assume $4|q$ and $m$ is odd. Then

        \begin{equation}\label{4: GaussSummsMultiplicativeEqn}
        S(a,m,q) = 0.
    \end{equation}
\end{lemma}
\begin{proof}
    Because of multiplicativity, it is sufficient to show the result with $q=2^k$ and $k\geq 2$. We have

    \begin{equation*}
        S(a,m,2^j) = \sum_{k=0}^{2^j-1}e(\frac{ak^2+mk}{2^j}).
    \end{equation*}
    Let us now compute what values $ak^2+mk$ takes as $k$ varies. Because $a$ and $m$ are odd, $ak^2+mk$ is even for every value of $k$. Moreover, suppose we have two solutions $k_1,k_2$ of

    \begin{equation*}
        ak^2+mk + c = 0  \,(\text{mod }2^j),
    \end{equation*}
    and that $k_1$ and $k_2$ are the same parity. Then

    \begin{equation*}
        ak_1^2-ak_2^2 + mk_1 - mk_2 = (k_1-k_2)(a(k_1+k_2)+m) = 0 \,(\text{mod }2^j).
    \end{equation*}
    The term $a(k_1+k_2)+m$ is always odd. This forces 

    \begin{equation*}
        (k_1-k_2) = 0\,(\text{mod }2^j).
    \end{equation*}
    Therefore, we can have, at most, one odd solution and one even solution mod $2^j$. As $k$ varies in $\{0,...,2^{j}-1\}$, the value $ak^2+mk$ runs through each even number in $\mathbb{Z}/2^j\Z{}$ exactly twice by basic counting. We can write the sum as

    \begin{equation*}
        S(a,m,2^j) = 2 \sum_{k=0}^{2^{j}-1}e(\frac{2k}{2^j}).
    \end{equation*}
    Because $j\geq 2$, the above sum is 0 by the geometric series formula.
    
\end{proof}

\begin{lemma}\label{4: QuadRecipGaussSumLemma}
    Suppose $aq+m$ is even. Then

        \begin{equation}\label{4: QuadRecipGaussSumEqn}
        |S(a,m,q)| = |\frac{q}{a}|^{-1/2}|S(-q,-m,a)|.
    \end{equation}
\end{lemma}
\begin{proof}
    This is Theorem 1.2.2 in \cite{GaussJacobiBook} with our normalization.
\end{proof}

These results are enough to prove \Cref{4: GaussSumBoundProp} for all natural numbers $q$.

\begin{proof}[Proof of \Cref{4: GaussSumBoundProp}]
    We always assume $(a,q)=1$. Because of multiplicativity, we may reduce to proving

    \begin{equation*}
        |S(a,m,p^j)|\leq p^{-j/2}
    \end{equation*}
    for all primes $p$  and integers $a,j$ such that $j\geq 1$ and $(a,q)=1$. In the case $p=2$, we allow a multiplied constant.
    
    Suppose $p$ is odd. Then $p^j$ is as well. In what follows, denote $x^*$ the multiplicative inverse of $x$ modulo $p^j$ and $(\frac{a}{p^j})$ the Jacobi symbol. Direct computation and completing the square gives

    \begin{align*}
        S(a,m,p^j) = & \frac{1}{p^j}e(-4^*a^*m^2/p^j)\sum_{k=0}^{p^j-1}e(a/p^j(k^2-2k2^*a^*m+4^*(a^*)^2m^2))\\
        &=\frac{1}{p^j}e(-4^*a^*m^2/p^j)\sum_{k=0}^{p^j-1}e(k^2a/p^j)\\
        &= e(-4^*a^*m^2/p^j)(\frac{a}{p^j})S(1,0,p^j).
    \end{align*}
    So the result follows in the case from the bound on the standard quadratic Gauss sum. Now suppose $p^j = 2^j$ for some j. Recall the bound for $S(1,0,2^j)$ follows from Corollary 1.2.3 in \cite{GaussJacobiBook}.
    
    The result is trivial if $j=1$. If $j\geq 2$, then $2^j$ is divisible by $4$. If $m$ is odd, then \Cref{4: OddForPower4Lemma} immediately gives the result as the sum is $0$. If $m$ is even, then we can apply \Cref{4: QuadRecipGaussSumLemma} with $q=2^j$ to get

    \begin{equation*}
        |S(a,m,2^j)| = |\frac{2^j}{a}|^{-1/2}|S(-2^j,-m,a)|.
    \end{equation*}
    The result then follows as $a$ is coprime to $2^j$, and is therefore odd, after another application of multiplicativity. 
\end{proof}
These bounds suffice for our results in \Cref{Section2}. In \Cref{Section3}, we sum over $a$ in the expression for $K_{Q,s}$, yielding $\mathcal{S}(m,q)$. Ostensibly, the sum over $a$ gives additional cancellation, however, it is difficult to leverage this to prove sharp results.

\subsection{Difficulty in improving the range of $p$}\label{4 SubSectionFailureToImprove}
We wish to use cancellation in the sum over $a$ to get an improved bound in $\mathcal{S}(m,q)$. For this purpose, we have the following lemma.
\begin{lemma}\label{3: singularSeriesLemma}
The function $\mathcal{S}(m,q)$ is multiplicative in $q$. If $q$ is odd, it has the following expression.

\begin{equation}\label{3: SingSeriesOddValueEqn}
     \mathcal{S}(m,q) = S(1,0,q)^d \sum_{(a,q)=1}\bigPara{\frac{a}{q}}^de((\nu/q)a^*-(\la/q)a),
\end{equation}
where $\nu=-4^*(m_1^2+...+m_d^2)$.
\end{lemma}
The proof is standard, but we do not present it as we will not prove results in this subsection.  As the Jacobi symbol takes on the values $0,1,-1$, the expression for $\mathcal{S}(m,q)$ reduces to an expression involving Kloosterman sums when $d$ is even and Sail\'e sums when $d$ is odd. We now recall the definitions of these sums.

\begin{definition}\label{3: SaileKloostermanSumsProp}
    We define a Kloosterman sum as

    \begin{equation}\label{KloostSumDefEqn}
        \textbf{Kl}(\alpha,\beta,q) \coloneq\sum_{(a,q)=1}e(\frac{a\alpha + a^*\beta}{q}).
    \end{equation}
    And the Sail\'e sums

    \begin{equation}\label{SalSumDefEqn}
        \textbf{Sa}(\alpha,\beta,q) \coloneq\sum_{(a,q)=1}(\frac{a}{q})e(\frac{a\alpha + a^*\beta}{q}).
    \end{equation}
\end{definition}
For $q$ a prime power, we have square root cancellation and each can be bounded by $\lesssim q^{1/2}$, see \cite{conradKloosterman2002} for example. For general $q$, the best bound we can hope for both these sums is

\begin{equation*}
    \textbf{Kl}(\alpha,\beta,q),\textbf{Sa}(\alpha,\beta,q) \lesssim \tau(q)\sqrt{\gcd(\alpha,\beta,q)}\sqrt{q}.
\end{equation*}

Here $\tau(q)$ is the number of divisors of $q$. In our case we have $\alpha = -\la$ and $\beta = \nu$. As we assume $(a,q)=1$, the bound simplifies in our case to

\begin{equation*}
    \textbf{Kl}(-\la,\nu,q),\textbf{Sa}(-\la,\nu,q) \lesssim q^\eps\sqrt{\gcd(\nu,\la,q)}\sqrt{q}.
\end{equation*}

This is the bound stated at the end of Section 3 in \cite{bourgainDemeterDiscreteRestrictImprovements2013}. The power of $q^\epsilon$ is inconsequential for our purposes. The factor $\sqrt{\gcd(\la,q)}$ is not. If $q|\la$, this means we get no cancellation in the sum using this bound, and we default to using the trivial bound $\lesssim q$ from the triangle inequality. However, both the Kloosterman and Sail\'e sums simplify in this case to sums with at least square root cancellation, although a different argument must be presented. There are additional cases to consider, however. $\la$ and $q$ could have a large common factor but $q$ not divide $\la$. In this case, we simplify to a sum that is not classical and the analysis is more complicated. There is also an additional summation over $m$ that may assist in proving better bounds. We intend to explore this in future work to prove both sharp bounds and bounds with an $\eps$-loss. 

A difficulty with improving the bound in \Cref{3: KQsFourierCoeffBoundProp} is the following. An implicit step in \cite{bourgainTorusEFunc1993} and \cite{bourgainDemeterDiscreteRestrictImprovements2013} is to restrict to $Q\lesssim N^{\eps_0}$ and use the triangle inequality.  One can prove an improved bound, replacing $Q^2$ by $Q$ in our proposition, for $Q\gtrsim N^{\eps_0}$ using essentially the same proof as for \Cref{2: KerrFourierBoundProp} and an additional step using a mean zero normalization. For small $Q$, the best we can hope for is the triangle inequality. The improved bound for moderate $Q\gtrsim N^{\eps_0}$ does not change our final bound as we are limited by the terms $Q\sim 1$.

Take $Q\sim \log \log N$, for example, and consider the integer formed by

\begin{equation*}
    X=\prod_{q\sim Q}q.
\end{equation*}
This integer satisfies $X\lesssim (\log \log N)^{\log \log N}\lesssim N$. The same approach we used to estimate \eqref{2: FourierTransNumberThrBoundEqn} will not work as we could have many factors that divide $l$. Therefore, we must use the trivial estimates near $Q\sim 1$, which restricts the range of $p$ we can apply our argument in. 

\begin{remark}
If we assume $\la$ does not have many divisors, we do get an improved bound. In the ideal situation, we have $\la$ being prime, and the techniques of this paper give the sharp results for $p>\frac{2(d+1)}{d-3}$ for $d\geq 5$, by using the full square root cancellation in Kloosterman and Sail\'e sums. Intermediate bounds can be proven based on the structure of the divisors of $\la$, but as we are interested in the strongest bound for the full set of eigenvalues we do not record them here.  
\end{remark}
 
\section{Comparison with result with loss}\label{SectionComparison}
In this section we will catalog the differences between our proofs and the proofs in \cite{bourgainDemeterDiscreteRestrictImprovements2013}.

A more careful integration by parts argument allows us to prove a no loss version of the dispersive type estimate for $G(t,x)$. This is already known by other methods when $q\leq N$. For the $\log$-loss argument we had to show the same holds near rationals with very small denominator.

A more careful analysis of $\mathcal{F}(K-K^Q)$ gives that the $Q^\eps$ bound there is in fact a $\log$ coming from the normalization constant $\log Q$. Such a factor is always present if we take our denominators to be primes. Taking the set of denominators to be all numbers eliminates this loss, but then the roots of unity argument used to simplify the expression \eqref{2: FourierTransNumberThrBoundEqn} no longer holds, and a loss is incurred there regardless. 

When proving the sharp result, our philosophy is to take any power of $Q$ as long as we do not have an $\eps$-loss in the parameters $N$ and $2^s$. This gives us worse bounds with respect to $p$ compared to previous works. Additionally, our intolerance of $\eps$-losses prevents us from improving over the trivial bound in a few places, most notably \eqref{3: S(m,q)TriangleIneqEqn} and \eqref{3: smallQKQsFourierCoeffBoundEqn}. It seems different techniques are needed to improve these bounds. 
\bibliographystyle{amsalpha}
\bibliography{zBib}
\Addresses
\end{document}